\documentclass[12pt]{article}
\usepackage{amsmath,amssymb,amsthm}

\usepackage{hyperref}
\newtheorem{theorem}{Theorem}[section]
\newtheorem{definition}[theorem]{Definition}
\newtheorem{cor}[theorem]{Corollary}

\newtheorem{lem}[theorem]{Lemma}

\newtheorem{pro}[theorem]{Proposition}
\DeclareMathOperator{\Real}{Re}
\usepackage{color}
\begin{document}
\title{\bf Evans-Selberg potential on planar domains\rm}

\author{Robert Xin DONG}
\date{}
\maketitle

\begin{abstract}
We provide explicit formulas of Evans kernels, Evans-Selberg potentials and fundamental metrics on potential-theoretically parabolic planar domains.
\end{abstract}

\renewcommand{\thefootnote}{\fnsymbol{footnote}}
\footnotetext{\hspace*{-7mm} 
\begin{tabular}{@{}r@{}p{16.5cm}@{}}
& 2010 Mathematics Subject Classification. Primary 30F15; Secondary 31A05, 31A15, 30F20\\
& Key words and phrases. Evans-Selberg potential, Evans kernel, potential-theoretically parabolic Riemann surface, Green function, Green kernel, fundamental metric
\end{tabular}}
\section{Introduction}
All open Riemann surfaces can be classified in the potential-theoretical sense into two types, namely hyperbolic ones and parabolic ones. The latter case happens if and only if there exists no Green function, or equivalently there exists no non-constant subharmonic function bounded from above. On a potential-theoretically parabolic Riemann surface, there exists a so-called Evans-Selberg potential, whose existence is equivalent to the parabolicity condition (see \cite{Ev, Ku, Na62, Se}). In contrast to various applications of Evans-Selberg potentials (see \cite{M, SV}), concrete examples are not quite understood. In this paper, on potential-theoretically parabolic planar domains, we provide explicit formulas of Evans-Selberg potentials, as well as formulas of the so-called Evans kernels and fundamental metrics, whose definitions will be recalled in the next section.

\begin{theorem} \label{Evans-Selberg-C-0} There exist Evans-Selberg potentials on $\mathbb C\setminus \{0\}$ with a pole $q$ given by $$e_{q}(p):=\log \frac{\left|p-q\right|}{\left|p\right|^k \left|q\right|^l},$$ where $k, l\in (0, 1).$ 
\end{theorem}

Using this $e_{q} (p)$, we derive a corollary as follows:

\begin{cor} \label{fund-C-0} There exist fundamental metrics on $\mathbb C\setminus \{0\}$ (in coordinate $z$) given by $$|z|^{-s}|dz|^2,$$ where $s\in (0, 2).$ \end{cor}

For the Evans kernels, we have the following theorem.

\begin{theorem} \label{Evans-C-0} There exist Evans kernels on $\mathbb C\setminus \{0\}$ given by $$e(p, q):=\log \frac{\left|p-q\right|}{\left|pq\right|^l},$$ where $l\in (0, 1).$ 
\end{theorem}

By the construction process of the above $e(p, q)$, we obtain the following result.

\begin{theorem} \label{Green-annulus} Given $t>0$, let $G_t(p,q)$ be the Green kernel on $\{z\in \mathbb C \,| e^{-2t}<|z|<e^{2t}\}$. Then, $$ \lim_{t \to +\infty}\left(G_t(p,q)+ \log\left(e^t-e^{-t}\right) \right)=\log \frac{\left|p-q\right|}{\sqrt{\left|pq\right|}},$$ uniformly on each compact subset of $\mathbb C\setminus \{0\}\times \mathbb C\setminus \{0\}$.
\end{theorem}

Similarly, we have results for the twice punctured complex plane, which is also potential-theoretically parabolic although it is hyperbolic in the Poincar\'e sense.

\begin{theorem} \label{Evans-Selberg-C-0-1} There exist Evans-Selberg potentials on $\mathbb C\setminus \{0,1\}$ with a pole $q$ given by $$\log \frac{|p-q|}{{\left|p\right|^{k}\left|q\right|^{l}} {\left|p-1\right|^{m} \left|q-1\right|^{n}}},$$ where $k, l, m, n>0$, $k+m<1$ and  $l+n<1.$ \end{theorem}

\begin{cor} There exist fundamental metrics on $\mathbb C\setminus \{0, 1\}$ (in coordinate $z$) given by $$|z|^{-s}|z-1|^{-j}|dz|^2,$$ where $s, j >0$ and $s+j<2.$ \end{cor} 

\begin{theorem} \label{Evans-C-0-1} There exist Evans kernels on $\mathbb C\setminus \{0,1\}$ given by $$\log \frac{|p-q|}{{|pq|^{k}} {|(p-1)(q-1)|^{m}}},$$ where $k, m>0$ and $k+m<1.$ \end{theorem}

\section{Preliminaries}

Let's first recall the Removable Singularity Theorem for a harmonic function (cf. \cite{B, Ro}).

\begin{pro} \label{RSTH} If $u$ is harmonic and bounded on the punctured disc $\{z\in \mathbb C: 0<|z|< 1\}$, then it extends to a harmonic function on the whole disc.
\end{pro}

Proposition \ref{RSTH} could follow from the following Proposition \ref{Bochner}, which describes the behaviors near isolated singularities (cf. \cite [p.50] {ABR}).

\begin{pro} [B$\hat o$chner's Theorem] \label{Bochner} Let $w\in D \subset \mathbb C$, and let $h$ be a positive harmonic function on $D-\{w\}$. Then $-h$ extends to be a subharmonic function on $D$, and there exists a harmonic function $k$ on $D$ and a constant $b \geq 0,$ such that 
$$h(z)=k(z)-b\log|z-w|,\ \,z\in D-\{w\}.$$
\end{pro}

For a subharmonic function, a generalized Removable Singularity Theorem is described as follows (see \cite [Theorem 3.6.1] {Ra}):

\begin{pro} \label{RSTSH} Let $U$ be an open subset of $\mathbb C$, let $E$ be a closed polar set, and let $u$ be a subharmonic function on $U-E$. Suppose that each point of $U\cap E$ has a neighborhood $N$ such that $u$ is bounded from above on $N-E$. Then $u$ has a unique subharmonic extension to the whole of $U$.
\end{pro}

Next, let's look at the definition of an Evans-Selberg potential (cf. \cite[p. 351]{SN}, \cite [p. 114]{SNo}). 

\begin {definition}\label{def-Evans}On an open Riemann surface $\Sigma$, an Evans-Selberg potential $E_{q}(p)$ with a pole $q\in \Sigma$ is a real-valued function satisfying the following conditions:

$(i)$ For all $p\in \Sigma\setminus \{q\}$, $E_{q}(p)$ is harmonic with respect to $p$,  

$(ii)$  $E_{q}(p)-\log|\varphi(p)-\varphi(q)|$ is bounded near $q$, with $\varphi$ being the local coordinate,

$(iii)$ ${ E_{q}(p)} \to +\infty$, as $p\to a_{\infty},$ the Alexandroff ideal boundary point of $\Sigma$.
\end {definition} 

Moreover, if this potential $E(p, q):=E_{q}(p)$ is symmetric in $(p, q)$ (regarded as a function on $\Sigma\times \Sigma$), and $E(p, q)-E(p, q^{\prime})$ are bounded near the boundary for any pair $(q, q^{\prime})$, then $E(p, q)$ is called an Evans kernel (see \cite{Na67}, \cite[p.354]{SN}). Two Evans kernels with the same prescribed singularities at the boundary are up to an additive constant by the Maximum Principle of subharmonic functions. Important properties of an Evans kernel are its joint continuity and uniform convergence, implying that it is approximable by Green kernels.
 
\begin {pro}[Nakai] \label{Nak} Let $E(p, q)$ be an Evans kernel on $\Sigma$, and $G_t(p,q)$ the negative Green kernel on $\Sigma_t:=\{p \in \Sigma \,| E(p, q_0)< t\}$ with a fixed $q_0 \in \Sigma$. Then \begin{equation} \label{EG} E(p, q)=\lim_{t \to +\infty}\left(G_t(p,q)+ t\right)\end{equation} uniformly on each compact subset of $\Sigma\times \Sigma$. \end {pro} 

Proposition \ref{Nak} is useful for computing some explicit formulas of Evans-Selberg potentials. Typical examples of parabolic planar domains are the complex plane $\mathbb C$, finitely-punctured complex planes, and $\mathbb C\setminus \mathbb Z$. Thus, it seems desirable to determine the Evans kernel by $\eqref{EG}$ as long as explicit formulas of Green kernels are known. Meanwhile, the above $\Sigma_t$ and $N_t$ are attainable in some special cases. Finally, we look at the definition of the so-called fundamental metric \cite{SV}.

\begin{definition} \label{fund} On a potential-theoretically parabolic Riemann surface $\Sigma$, the fundamental metric under the local coordinate $z=\varphi(p)$ is defined as \begin{equation*}\label{fundamental metric}
c(z)|dz|^2:=\exp \lim_{q \to p}\left(E_q(p)-\log|\varphi(p)-\varphi(q)|\right)|dz|^2.
\end{equation*} \end{definition}

\section{Explicit formulas} Explicit formulas of the Evans-Selberg potentials are not quite understood, except the case of $\mathbb C$ where the logarithmic kernel $\log |p-q|$ becomes a good candidate. In this section, we provide explicit formulas of the Evans-Selberg potential on punctured complex planes.

\begin{proof} [Proof of Theorems \ref{Evans-C-0} and \ref{Green-annulus}] Without loss of generality assume $q_0=1$ in Proposition 2.2. For any $t>1,$ choose a function $r=r(t)>0$ (to be determined later) such that $r \searrow 0^+$ as $t\to +\infty$ and the annulus $A_r:=\{p\in \mathbb C \, \left| {r}<|p|<{1}/{r}\right.\}$ is equal to $R_T:=\{p \in \Sigma \,| E(p, 1)<\log\left(e^t-e^{-t}\right):=T\}$, which admits a Green kernel $g_T(p,q)$. By \cite[p. 386--388] {CH}, the negative Green kernel for $A_r$ is 

\begin{align*}g_T(&p,q)=\Real \left\{\log \left(r^{\frac{1}{2}} \cdot z^{\frac{-\log q}{2\log r}} \left(\sqrt{\frac{p}{q}}-\sqrt{\frac{q}{p}}\right) \cdot \frac{\prod _{j=1}^{\infty}\left(1-\frac{p}{q}\cdot r^{4j}\right)\left(1-\frac{q}{p}\cdot r^{4j}\right)}{\prod_{j=1}^{\infty}\left(1-pq\cdot r^{4j-2}\right)\left(1-\frac{1}{pq}\cdot r^{4j-2}\right)}\right)\right\}.\end{align*}

By Proposition \ref{Nak}, it follows that  \begin{align*}
E(p, q)= &\lim_{T \to \infty}  \Real \left\{\log \left(r^{\frac{1}{2}} \cdot z^{\frac{-\log q}{2\log r}} \cdot \left(\sqrt{\frac{p}{q}}-\sqrt{\frac{q}{p}}\right)\right) +T\right\}\\
=& \lim_{T \to \infty} \left\{\frac{1}{2} \log r + \log \frac{|p-q|}{\sqrt{|pq|}} +T\right\}. \end{align*}

Choosing $r$ such that \begin{equation}\label{N_t} \lim_{T \to \infty} \left (\frac{1}{2} \log r + T \right)=0,\end{equation} we know that the Evans kernel becomes $$E(p, q)=\log \frac{|p-q|}{\sqrt{|pq|}}.$$ 

It is easy to check that $E(p, q)-E(p, q^{\prime})$ are bounded near the boundary (consisting points $0$ and $\infty$) for any pair $(q, q^{\prime})\in \mathbb C\setminus \{0\}\times \mathbb C\setminus \{0\}$, $E(p, q)$ is symmetric in $(p, q)$, and $E(p, q)$ tends to $+\infty$ at the boundary. Finally, setting $r(t):=e^{-2t}$, we know that they satisfy \eqref{N_t} by definition. Moreover, when $t$ is sufficiently large it holds that $A_r=R_T$, i.e., $$\left \{ r<|p|<\frac{1}{r}\right \}=\left \{ \frac{|p-1|}{\sqrt{|p|}}< e^t-e^{-t} \right \}.$$ 

Thus, Theorem \ref{Green-annulus} is proved. For Theorem \ref{Evans-C-0}, it suffices to check by definition that for any fixed $0<l<1$, $e(p, q)$ gives an Evans kernel on $\mathbb C\setminus \{0\}$. \end{proof}

By dropping the symmetry in $(p, q)$, one can easily construct Evans-Selberg potentials on $\mathbb C\setminus \{0\}$, which gives Theorem \ref{Evans-Selberg-C-0}. Corollary \ref{fund-C-0} then follows from Theorem \ref{Evans-Selberg-C-0} and Definition \ref{fund}. For the case of $\mathbb C\setminus \{0,1\}$, we are not sure how to make the approximation process and to use Proposition \ref{Nak}. Nevertheless, by the formulas in Theorems \ref{Evans-Selberg-C-0} and \ref{Evans-C-0}, we can construct by hand Evans kernels and Evans-Selberg potentials, which yield Theorems \ref{Evans-Selberg-C-0-1} and \ref{Evans-C-0-1}, respectively. Via elliptic functions, the author in \cite{D} constructed an Evans-Selberg potential on a once-punctured complex torus.

\section{Boundary behaviors} 

In this section, we first fix a point $q_0 \in \mathbb C\setminus \{0\}$. For an arbitrary Evans-Selberg potential $E_{q_0} (p)$ with a pole $q_0$, let $D$ be a neighborhood of $0$ such that $E_{q_0} (p)>0$ on $D$. Then, according to Proposition \ref{Bochner} (B$\hat o$chner's Theorem), there exists a constant $b_{0}\geq 0$ such that $ E_{q_0}(p)+b_{0}\cdot \log{|p|}$ is harmonic and bounded on $D\ni 0$. Similarly, for this same $E_{q_0}(p)$, there exists a constant $b_{\infty}\geq 0$ such that $ E_{q_0}(p)-b_{\infty}\cdot \log{|p|}$ is harmonic and bounded near $\infty$. Denote $\max \{b_0, b_{\infty}\}$ by $b_{\max,\,q_0}$ or just $b_{\max}$ (if $q_0$ is not necessarily specified).

\begin{lem} For any fixed ${q_0}\in \mathbb C\setminus \{0\}$, $$\inf\limits_{E_{q_0}(p)} \left\{b_{\max,\,q_0} \right\}=\frac{1}{2}.$$
\end{lem}

\begin{proof} First, denote the above left hand side by $b_{\mathbb C\setminus \{0\}}\equiv \inf
\left\{b_{\max,\,q_0} \right\}.$ For fixed $q_0 \in \mathbb C\setminus \{0\}$, from $e_{q_0}(p)$ (for $k=l=1/2$) in Theorem \ref{Evans-Selberg-C-0}, we know that $b_{\mathbb C\setminus \{0\}} \leq {1}/{2}.$ Now, let us assume that $b_{\mathbb C\setminus \{0\}} < {1}/{2}$ and conduct the proof by contradiction. Then, for an arbitrary Evans-Selberg potential $E_ {q_0}(p)$ with a pole ${q_0}$, it holds that $$\overline{\lim\limits_{p\to 0}} \left \{E_{q_0} (p) + b_{\mathbb C\setminus \{0\}} \cdot \log{|p|}\right\}<\infty,$$ and $$\overline{\lim\limits_{p\to \infty}} \left \{E_{q_0} (p)- b_{\mathbb C\setminus \{0\}} \cdot  \log{|p|}\right\}<\infty.$$ Therefore, $E_{q_0}(p)-e_{q_0}(p)<E_{q_0}(p)+ b_{\mathbb C\setminus \{0\}}  \cdot \log|p|-{1}/{2} \cdot \log|p|-e_{q_0}(p)<\infty$, as $p\to 0.$ Similarly, $E_{q_0}(p)-e_{q_0}(p)<\infty$, as $p\to \infty.$ Since $E_{q_0}(p)-e_{q_0}(p)$ is harmonic in $p$ on $\mathbb C\setminus \{0\}$ and bounded from above, according to Proposition \ref{RSTSH}, it extends as a subharmonic function on $\hat{\mathbb C}:=\mathbb C\cup\{\infty\}$. Therefore, $E_{q_0}-e_{q_0}$ must be a constant, which is a contradiction to $b_{\mathbb C\setminus \{0\}} < {1}/{2}$. So, the lemma is proved.
\end{proof} 

Since the final result in the above Lemma does not depend on any particular choice of $q_0$, we further obtain the following theorem for any Evans-Selberg potential $E_q(p)$ on $\mathbb C\setminus \{0\}$ with a pole $q$.

 \begin{theorem}  $$\inf\limits_{E_{q}(p)} \left\{b_{\max,\,q}: q \in \mathbb C\setminus \{0\} \right\}=\frac{1}{2}.$$
\end{theorem}

For the domain $\mathbb C\setminus\{0,1\}$, by a similar argument, we can find constants $b_{0}, b_{1}, b_{\infty}\geq 0$ such that $E_{q}(p)+b_{0}\cdot \log{|p|}$, $E_{q}(p)+b_{1}\cdot \log{|p-1|}$, and $E_{q}(p)-b_{\infty}\cdot \log{|p|}$ are all harmonic and bounded from above near $0$, $1$, and $\infty$, respectively. Denote $\max \{b_0, b_1, b_{\infty}\}$ by $b_{\max}$. Then, we have the following result for any Evans-Selberg potential $E_q(p)$ on $\mathbb C\setminus \{0, 1\}$ with a pole $q$.

 \begin{theorem} $$\inf\limits_{E_{q}(p)} \left\{b_{\max,\,q}: q \in \mathbb C\setminus \{0, 1\} \right\}=\frac{1}{3}.$$
\end{theorem}

\subsection*{Acknowledgements} The author thanks Profs T. Ohsawa and M. Adachi for the fruitful discussions. He also thanks Prof. H. Yamaguchi for the very valuable comments. This work is supported by the Ideas Plus grant 0001/ID3/2014/63 of the Polish Ministry of Science and Higher Education, KAKENHI and the Grant-in-Aid for JSPS Fellows (No. 15J05093).

\fontsize{10}{11}\selectfont

\bigskip

Instytut Matematyki, Uniwersytet Jagiello\'nski, \L{}ojasiewicza 6, 30-348 Krak\'ow, Poland;\\
Email: 1987xindong@tongji.edu.cn

\end{document}